\documentclass[paper=a4,english,fontsize=11pt,parskip=half,abstract=true]{scrartcl}
\usepackage{babel}
\usepackage[utf8]{inputenc}
\usepackage[T1]{fontenc}
\usepackage[left=20mm,right=20mm,top=30mm,bottom=30mm]{geometry}
\usepackage{amsmath}
\usepackage{amsthm}
\usepackage{amssymb}
\usepackage{enumerate} %superseeded by enumitem
\usepackage{thmtools}
\usepackage{mathtools}
\mathtoolsset{centercolon} %makes := symmetric
\usepackage[bookmarks=true,
            pdftitle={Groups with 2-generated Sylow subgroups and their character tables},
            pdfauthor={Benjamin Sambale},
            pdfkeywords={},
            pdfstartview={FitH}]{hyperref}

\newtheorem{Thm}{Theorem} %[section]
\newtheorem{Lem}[Thm]{Lemma}
\newtheorem{Prop}[Thm]{Proposition}
\newtheorem{Cor}[Thm]{Corollary}

\theoremstyle{definition}

\newtheorem*{ThmA}{Theorem A}
\newtheorem*{ThmB}{Theorem B}
\newtheorem*{ThmC}{Theorem C}
\newtheorem*{CorD}{Corollary D}
\numberwithin{equation}{section}

\allowdisplaybreaks[1]

\renewcommand{\phi}{\varphi}
\newcommand{\C}{\mathrm{C}}
\newcommand{\N}{\mathrm{N}}
\newcommand{\Z}{\mathrm{Z}}
\newcommand{\pcore}{\mathrm{O}}
\newcommand{\core}{\mathrm{core}}
\newcommand{\E}{\mathrm{E}}
\newcommand{\F}{\mathrm{F}}
\newcommand{\K}{\mathrm{K}}
\newcommand{\ZZ}{\mathbb{Z}}

\newcommand{\QQ}{\mathbb{Q}}

\newcommand{\FF}{\mathbb{F}}
\newcommand{\FS}{\mathcal{F}}
\newcommand{\Aut}{\operatorname{Aut}}
\newcommand{\Inn}{\operatorname{Inn}}
\newcommand{\Out}{\operatorname{Out}}
\newcommand{\GL}{\operatorname{GL}}
\newcommand{\SL}{\operatorname{SL}}
\newcommand{\PSU}{\operatorname{PSU}}
\newcommand{\PSL}{\operatorname{PSL}}
\newcommand{\PGL}{\operatorname{PGL}}

\newcommand{\Sz}{\operatorname{Sz}}
\newcommand{\Irr}{\operatorname{Irr}}

\newcommand{\Gal}{\operatorname{Gal}}
\newcommand{\Syl}{\operatorname{Syl}}

\newcommand{\Ker}{\operatorname{Ker}}

\newcommand{\tr}{\operatorname{tr}}

\newcommand{\TT}{\mathrm{t}}

\title{Groups with 2-generated Sylow subgroups\\ and their character tables}
\author{Alexander Moretó\footnote{Departament de Matem\`atiques, Universitat de Val\`encia, 46100 Burjassot, Val\`encia, Spain, \href{mailto:alexander.moreto@uv.es}{alexander.moreto@uv.es}} \ and Benjamin Sambale\footnote{Institut für Algebra, Zahlentheorie und Diskrete Mathematik, Leibniz Universität Hannover, Welfengarten 1, 30167 Hannover, Germany,
\href{mailto:sambale@math.uni-hannover.de}{sambale@math.uni-hannover.de}}}
\date{\today}

\begin{document}
\frenchspacing
\maketitle
\begin{abstract}\noindent
Let $G$ be a finite group with Sylow $p$-subgroup $P$. We show that the character table of $G$ determines whether $P$ has maximal nilpotency class and whether $P$ is a minimal non-abelian group. The latter result is obtained from a precise classification of the corresponding groups $G$ in terms of their composition factors. For $p$-constrained groups $G$ we prove further that the character table determines whether $P$ can be generated by two elements.
\end{abstract}


\renewcommand{\sectionautorefname}{Section}
\section{Introduction}

Recently, Navarro and the second author \cite{NS22} have investigated finite groups $G$ with a Sylow $p$-subgroup $P$ such that  $|P:P'|=p^2$ or $|P:\Z(P)|=p^2$ where $P'=[P,P]$ denotes the commutator subgroup and $\Z(P)$ is the center of $P$. It was proved that both properties can be read off from the character table $X(G)$ of $G$. This was another contribution to Richard Brauer's Problem 12 of
\cite{BrauerLectures}, which asks what properties of a Sylow $p$-subgroup $P$ are determined by $X(G)$. We refer the reader to the introduction of \cite{NS22} and \cite{sam20} for a collection of the known results on this problem. We just mention that one important property is that $X(G)$ knows whether $P$ is abelian.
While there is an elementary proof of the case  $p=2$ by Camina--Herzog~\cite{CaminaHerzog}, the full solution has required the classification of finite simple groups (see \cite{KimmerleSandling,NST,MN}).

After dealing with $P'$ and $\Z(P)$, it is natural to turn our attention to the Frattini subgroup $\Phi(P)$ of $P$. Recall that $|P:\Phi(P)|\le p$ holds if and only if $P$ is cyclic. It is easy to show that this property can be read off from $X(G)$ (see \cite[Corollary~3.12]{Nav2}).
In the first part of the present paper we consider groups $G$ with $|P:\Phi(P)|=p^2$, i.\,e. $P$ is generated by two elements, but not by one. For $p=2$ this property is detectable by $X(G)$ as was shown by Navarro et al.~\cite{NRSV}. We obtain the corresponding result for odd primes $p$ provided that $G$ is $p$-constrained in \autoref{corpsolv}. In the general case we offer a partial solution depending on the socle of $G$ (see \autoref{lemred} and the subsequent remark).

Our next objective are groups with Sylow $p$-subgroups $P$ of maximal nilpotency class.
For $p=2$, this property is equivalent to $|P:P'|=4$. This case was previously handled in an elementary fashion by Navarro--Sambale--Tiep~\cite{NaSaTi}. The general result is our first main theorem.

\begin{ThmA}
The character table of a finite group $G$ determines whether $G$ has Sylow $p$-subgroups of maximal nilpotency class.
\end{ThmA}

It is known that $X(G)$ determines the isomorphism types of abelian Sylow subgroups. Of course we cannot expect this for maximal class Sylow subgroups as $X(D_8)=X(Q_8)$. Perhaps surprisingly, $X(G)$ does not even determine $X(P)$. Counterexamples for $p=3$ arise as semidirect products of non-equivalent faithful actions of $\SL(2,3)$ on $C_9\times C_9$ (the groups are $\mathtt{SmallGroup}(2^33^5,a)$ with $a\in\{2289,2290\}$ in GAP~\cite{GAP48}). Here $P$ indeed has maximal class. This is related to \cite[Question~E]{NRS}.

We obtain Theorem~A as a consequence of the following structure description, which might be of independent interest.

\begin{ThmB}
Let $G$ be a finite group with a Sylow $p$-subgroup $P$ of maximal class. Suppose that $\pcore_{p'}(G)=1$ and $\pcore^{p'}(G)=G$. Then one of the following holds:
\begin{enumerate}[(i)]
\item There exists $x\in P$ such that $|\C_G(x)|_p=p^2$.
\item $G$ is quasisimple and $|\Z(G)|\le p$.
\end{enumerate}
\end{ThmB}

The proof uses recent work by Grazian--Parker~\cite{GrazianParker} on fusion systems and is given in \autoref{secmax}.

In the final part of the paper we study groups with minimal non-abelian Sylow $p$-subgroups $P$, i.\,e. $P$ is non-abelian, but every proper subgroup of $P$ is abelian. It is easy to see that this happens if and only if $|P:\Z(P)|=|P:\Phi(P)|=p^2$ (see \autoref{lemmna} below). Refining \cite[Theorem~7.5]{NS22}, we obtain in \autoref{secmna} the following description:

\begin{ThmC}
Let $G$ be a finite group with a minimal non-abelian Sylow $p$-subgroup $P$ and $\pcore_{p'}(G)=1$. Then one of the following holds:
\begin{enumerate}[(i)]
\item $p=2$, $P\in\{D_8,Q_8\}$ and $\pcore^{2'}(G)\in\{\SL(2,q),\PSL(2,q'),A_7\}$ where $q\equiv \pm 3\pmod{8}$ and $q'\equiv \pm 7\pmod{16}$.
\item $|P|=p^3$ and $\exp(P)=p>2$.
\item $G=P\rtimes Q$ where $Q\le\GL(2,p)$.
\item $p> 2$, $\pcore^{p'}(G)=S\rtimes C_{p^a}$ where $S$ is a simple group of Lie type with cyclic Sylow $p$-subgroups. The image of $C_{p^a}$ in $\Out(S)$ has order $p$.

\item $p=2$ and $G=\PSL(2,q^f)\rtimes C_{2^ad}$ where $q$ is a prime, $q^f\equiv\pm3\pmod{8}$ and $d\mid f$. Moreover, $C_{2^a}$ acts as a diagonal automorphism of order $2$ on $\PSL(2,q^f)$ and $C_d$ induces a field automorphism of order $d$.

\item $p=3$ and $\pcore^{3'}(G)=\PSL^\epsilon(3,q^f)\rtimes C_{3^a}$ where $\epsilon=\pm1$, $q$ is a prime, $(q^f-\epsilon)_3=3$ and $G/\pcore^{3'}(G)\le C_f\times C_2$.
\end{enumerate}
\end{ThmC}

Here, $\PSL^\epsilon$ stands for $\PSL$ if $\epsilon=1$ and $\PSU$ otherwise.
Again the proof is based on the classification of the corresponding fusion systems. To show that Case~\eqref{meta2} in Theorem~C occurs for all odd primes $p$, we will exhibit appropriate examples after the proof.

\begin{CorD}
The character table of a finite group $G$ determines whether $G$ has minimal non-abelian Sylow $p$-subgroups.
\end{CorD}

\section{\texorpdfstring{$2$}{2}-generated Sylow subgroups}\label{sec2gen}

In the following $G$ will always denote a finite group. The exponent of $G$ is denoted by $\exp(G)$. The core of a subgroup $H\le G$ is defined by $\core_G(H):=\bigcap_{g\in G}gHg^{-1}\unlhd G$.
For $x,y\in G$ let $[x,y]:=xyx^{-1}y^{-1}$.
The Fitting subgroup and the generalized Fitting subgroup of $G$ are denoted by $\F(G)$ and $\F^*(G)=\F(G)*\E(G)$ respectively.
For $g\in G$ and $\chi\in\Irr(G)$ let
\begin{align*}
\QQ(g)&:=\QQ(\chi(g):\chi\in\Irr(G)),\\
\QQ(\chi)&:=\QQ(\chi(g):g\in G).
\end{align*}
It is well-known that $\QQ(\chi)$ lies in the cyclotomic field $\QQ_n$ where $n=|G|$. Let $f_\chi$ be the smallest positive integer such that $\QQ(\chi)\subseteq\QQ_{f_\chi}$ ($f_\chi$ is called the \emph{Feit number} in \cite{Nav2}). Let $\Irr_{p'}(G):=\{\chi\in\Irr(G):p\nmid \chi(1)\}$ as usual. The $p$-part and the $p'$-part of an integer $n$ are denoted by $n_p$ and $n_{p'}$ respectively.

Our first lemma is applied multiple times throughout the paper.

\begin{Lem}\label{lemabel}
Let $A$ be an abelian normal subgroup of $G$ such that $G=\langle x\rangle A$ for some $x\in G$. Then the map $A\to G'$, $a\mapsto[x,a]$ is an epimorphism with kernel $\C_A(x)$. In particular, $|G'|=|A/\C_A(x)|$.
\end{Lem}
\begin{proof}
See \cite[Lemma~4.6]{IsaacsGroup}.
\end{proof}

To get from $P'$ to $\Phi(P)$ we need the following variant.

\begin{Lem}\label{lemPhi}
Let $P$ be a $p$-group with a proper normal subgroup $Q$ and $x\in P$ such that $P=\langle x\rangle Q$ and $\langle x\rangle\cap Q\le P'$. Then $|P:\Phi(P)|=p^2$ if and only if $|\C_{Q/\Phi(Q)}(x)|=p$.
\end{Lem}
\begin{proof}
Since $\langle x\rangle\cap Q\le P'\le\Phi(P)$ and $Q<P$, we have
\[P/\Phi(P)=Q\Phi(P)/\Phi(P)\times \langle x\rangle\Phi(P)/\Phi(P)\cong Q/(Q\cap\Phi(P))\times C_p.\]
Moreover,
\[\Phi(P)\cap Q=P'\Phi(Q)\langle x^p\rangle\cap Q=P'\Phi(Q)(\langle x^p\rangle\cap Q)=P'\Phi(Q).\]
Now $|P:\Phi(P)|=p^2$ if and only if
\[|Q/\Phi(Q):(P/\Phi(Q))'|=|Q:P'\Phi(Q)|=p.\]
By \autoref{lemabel} applied to $Q/\Phi(Q)\unlhd P/\Phi(Q)$, this is equivalent to $|\C_{Q/\Phi(Q)}(x)|=p$.
\end{proof}

The next result is a variation of \cite[Theorem~6.1]{NS22}.

\begin{Lem}\label{core}
Let $G$ be a finite group with Sylow $p$-subgroup $P$ and $\pcore_{p'}(G)=1$. Then
\[
K:=\bigcap_{\substack{\chi\in\Irr_{p'}(G)\\p^2\,\nmid\, f_\chi}}\Ker(\chi)=\core_G(\Phi(P)).
\]
\end{Lem}
\begin{proof}
Let $n:=|G|$. If $n_p=1$, then the claim holds since $\bigcap_{\chi\in\Irr(G)}\Ker(\chi)=1=P$. Thus, let $n_p\ne 1$. Then $\mathcal{G}:=\Gal(\QQ_n|\QQ_{pn_{p'}})$ is a $p$-group. Let $N:=\core_G(\Phi(P))$ and $\chi\in\Irr_{p'}(G)$ with $p^2\nmid f_\chi$. Since $\QQ(\chi_P)\subseteq\QQ(\chi)\subseteq\QQ_{pn_{p'}}$, $\mathcal{G}$ permutes the irreducible constituents of $\chi_P$. Since the sizes of the $\mathcal{G}$-orbits are $p$-powers and $p\nmid \chi(1)$, there must be a linear constituent $\lambda\in\Irr(P|\chi)$ fixed by $\mathcal{G}$, i.\,e. $\QQ(\lambda)\subseteq\QQ_p$. It follows that $N\subseteq\Phi(P)\subseteq\Ker(\lambda)$. By Clifford theory, $\chi_N$ is a sum of conjugates of $\lambda_N$. Hence, $N\subseteq\Ker(\chi)$. This shows that $N\le K$.

Now let $\lambda\in\Irr(P/\Phi(P))$. This time, $\mathcal{G}$ acts on the irreducible constituents of $\lambda^G$. Since $p\nmid |G:P|=\lambda^G(1)$, there must be a constituent $\chi\in\Irr_{p'}(G|\lambda)$ fixed by $\mathcal{G}$, i.\,e. $p^2\nmid f_\chi$. This implies $\chi_{P\cap K}=\chi(1)1_{P\cap K}$. On the other hand, $\lambda_{P\cap K}$ is a constituent of $\chi_{P\cap K}$. Therefore, $P\cap K\subseteq\Ker(\lambda)$. Since $\lambda\in\Irr(P/\Phi(P))$ was arbitrary, we obtain $P\cap K\le\Phi(P)$.
Now Tate's theorem (see \cite[Satz~IV.4.7]{Huppert}) yields that $K$ is $p$-nilpotent. By hypothesis, $\pcore_{p'}(K)\le\pcore_{p'}(G)=1$ and $K$ is a $p$-group. Finally, $K\le \pcore_p(G)\cap K\le P\cap K\le\Phi(P)$ and $K\le N$.
\end{proof}

We mention that the characters $\chi$ with $p^2\nmid f_\chi$ are precisely the \emph{almost $p$-rational} characters introduced in \cite{HMM}.
\autoref{core} allows to read off $K:=\core_G(\Phi(P))$ from the character table. Since $|P/K:\Phi(P/K)|=|P:\Phi(P)|$, it is therefore no loss to assume that $K=1$. The next theorem comes close to \cite[Theorem~3.1]{NS22}.

\begin{Thm}\label{2gen}
Let $G$ be a finite group with a non-abelian Sylow $p$-subgroup $P$ such that $|P:\Phi(P)|=p^2$ and $\pcore_{p'}(G)=1=\core_G(\Phi(P))$.
Then $\F^*(G)$ is the unique minimal normal subgroup of $G$ and $P\F^*(G)/\F^*(G)$ is cyclic.
If $\F^*(G)$ is non-abelian, then $P$ permutes the simple components of $\F^*(G)$ transitively. In particular, their number is a $p$-power in this case.
\end{Thm}
\begin{proof}
Let $N$ be a minimal normal subgroup of $G$. Then
\begin{align*}
|PN/N:\Phi(PN/N)|&=|P/P\cap N:\Phi(P/P\cap N)|=|P/P\cap N:\Phi(P)(P\cap N)/P\cap N|\\
&=|P:\Phi(P)(P\cap N)|\le|P:\Phi(P)|=p^2.
\end{align*}
Suppose first that $P\cap N\le\Phi(P)$. Then by Tate's theorem (see \cite[Satz~IV.4.7]{Huppert}), $N$ is a $p$-group and $N\le\Phi(P)$. This contradicts $\core_G(\Phi(P))=1$. Consequently, $|PN/N:\Phi(PN/N)|\le p$ and $PN/N$ is cyclic.
Let $M\ne N$ be another minimal normal subgroup of $G$. Then $G/N$ and similarly $G/M$ have cyclic Sylow $p$-subgroups.
Since $G$ is isomorphic to a subgroup of $G/M\times G/N$, $G$ has abelian Sylow $p$-subgroups, which we have excluded explicitly.
This shows that $N$ is the unique minimal normal subgroup.

Assume now that $N$ is non-abelian. Then $\F(G)\cap N=1$ implies $\F(G)=1=\Z(G)$ and $\F^*(G)=\E(G)=N$.
Write $N=T_1\times\ldots\times T_n$ with non-abelian simple groups $T_1\cong\ldots\cong T_n$.
If $P\le N$, then $n=1$ and $P$ certainly acts transitively on $\{T_1,\dots,T_n\}$. Hence, we may assume that $P\nsubseteq N$ and $n\ge 2$.
Let $Q_i:=P\cap T_i$ for $i=1,\ldots,n$. Let $x\in P$ such that $PN/N=\langle xN\rangle$.
Since $P\cap N\nsubseteq\Phi(P)$, there exists some $1\le i\le n$ with $Q_i\nsubseteq\Phi(P)$. Without loss of generality, let $i=1$. Choose $y\in Q_1\setminus\Phi(P)$. For all $j\in\ZZ$ we note that $xy^j\notin N\supseteq\Phi(P)$. Since $|P:\Phi(P)|=p^2$, it follows that $P=\langle x,y\rangle$. Without loss of generality, let $T_1,\ldots,T_k$ be the orbit of $T_1$ under $P$. Suppose by way of contradiction that $k<n$. Then $Q_1\ldots Q_k\unlhd P$ and $Q_{k+1}\times\ldots \times Q_n\le P/Q_1\ldots Q_k=\langle xQ_1\ldots Q_k\rangle$ is cyclic. This is only possible if $n=k+1$ and $Q_n$ is cyclic.
Moreover, $Q_n=\langle x^{p^a}z\rangle$ for some $a\ge 1$ and $z\in Q_1\cdots Q_k$. Since a non-abelian  simple group cannot have a cyclic Sylow $2$-subgroup, $p>2$. It follows from \cite[Theorem~A]{Gross} that $x$ induces an inner automorphism on $T_n$. This is impossible since $x^{p^a}$ induces an inner automorphism of order $|T_n|_p$. This contradiction shows that $P$ permutes the $T_i$ transitively.

Finally, assume that $N$ is elementary abelian. Since $\pcore_{p'}(G)=1$, we have $F:=\F(G)=\pcore_p(G)$. Suppose that $N<F$. Then $\Phi(F)\le\Phi(P)$ yields $\Phi(F)\le\core_G(\Phi(P))=1$, i.\,e. $F$ is elementary abelian.
Now the existence of an element of order $p$ in $P\setminus N$ implies the existence of a (cyclic) complement of $N$ in $P$.
By a theorem of Gaschütz (see \cite[Hauptsatz~I.17.4]{Huppert}), $N$ has a complement $K$ in $G$. Since $F$ centralizes $N$, we obtain $1\ne K\cap F\unlhd NK=G$. This contradicts the fact that $N$ is the unique minimal normal subgroup of $G$. Hence, $F=N$. Suppose that $\E(G)\ne 1$ and choose a central product $M\unlhd G$ of quasisimple components. Then $N\le\Z(M)$, because $1\ne N\cap M\unlhd G$. Since $M/N$ has cyclic Sylow $p$-subgroups, the order of the Schur multiplier of $M/N$ is not divisible by $p$. This contradicts $N\le\Z(M)$. We have therefore shown that $N=\F^*(G)$.
\end{proof}

In order to decide whether $|P:\Phi(P)|=p^2$, we may assume that the hypotheses of \autoref{2gen} are fulfilled. The situation now splits into two cases. When $\F^*(G)$ is abelian, the group $G$ is $p$-constrained (recall that in general a group $G$ is called $p$-constrained if $\C_{\overline{G}}(\pcore_p(\overline{G}))\le\pcore_p(\overline{G})$ where $\overline{G}:=G/\pcore_{p'}(G)$).
In this case we solve the problem completely.
To so do, we will use a result of Higman (see \cite[Corollary~7.18]{Nav2}) that allows to locate the $p$-elements in $X(G)$.

\begin{Cor}\label{corpsolv}
The character table of a $p$-constrained group $G$ determines whether a Sylow $p$-subgroup $P$ is generated by two elements.
\end{Cor}
\begin{proof}
Let $P$ be a Sylow $p$-subgroup of $G$. Since the character table $X(G)$ determines $X(G/\pcore_{p'}(G))$, we may assume that $\pcore_{p'}(G)=1$. Since $G$ is $p$-constrained, $\pcore_p(G)>1$. By \autoref{core}, we may assume that $\core_G(\Phi(P))=1$.
Moreover, the orders and embeddings of the normal subgroups of $G$ can be read off from $X(G)$. Hence by \autoref{2gen}, we may assume that $N=\pcore_p(G)=\F(G)$ is the only minimal normal subgroup of $G$. If $P=N$, then $|P:\Phi(P)|=|P|$ and we are done. Hence, let $N<P$.
By \cite[Corollary~3.12]{Nav2}, $X(G/N)$ detects whether $P/N$ is cyclic. By \autoref{2gen}, we can assume that this is the case. Choose $x\in P$ with $P/N=\langle xN\rangle$ (note that $x$ can be spotted in $X(G)$ using \cite[Corollary~3.12]{Nav2}). Since $P=N\langle x\rangle=\pcore_p(G)\langle x\rangle$ is the only Sylow $p$-subgroup of $G$ containing $x$, $\C_P(x)=\C_N(x)\langle x\rangle$ is a Sylow $p$-subgroup of $\C_G(x)$. In particular, $|\C_N(x)|=\frac{|\C_G(x)|_p}{|P/N|}$ is determined by $X(G)$. By \autoref{lemabel}, we have \begin{equation}\label{eqP}
P'=[x,N]=\{[x,y]:y\in N\}
\end{equation}
and $|P'|=|N/\C_N(x)|$ can be computed from $X(G)$. Let $|P/N|=p^a$ and $|N/P'|=p^n$. If $x^{p^a}\in P'$, then $P/P'\cong C_{p^a}\times C_p^n$ and otherwise $P/P'\cong C_{p^{a+1}}\times C_p^{n-1}$.
Since $\QQ(x)$ can be read off from $X(G)$, it suffices to show that
\[p|\QQ(x):\QQ|_p=\exp(P/P').\]
Taking only $X(G/N)$ into account, we obtain $\QQ(xN)=\QQ_{p^a}$ or equivalently $|\QQ(xN):\QQ|_p=p^{a-1}$ by \cite[Theorem~3.11]{Nav2}. It follows that $|\QQ(x):\QQ|_p\ge p^{a-1}$.
If $x^{p^a}=1$, then $p|\QQ(x):\QQ|_p=p^a=\exp(P/P')$ as desired. Now let $|\langle x\rangle|=p^{a+1}$. If $x^{p^a}\in P'$, then there exists $y\in N$ with $x^{p^a}=[x,y]=xyx^{-1}y^{-1}$ by \eqref{eqP}. It follows that $yxy^{-1}=x^{1-p^a}$ and $|\N_G(\langle x\rangle):\C_G(x)|_p=p$. Again by \cite[Theorem~3.11]{Nav2}, we have $p|\QQ(x):\QQ|_p=p^a=\exp(P/P')$. Assume conversely that $|\QQ(x):\QQ|_p=p^{a-1}$. Then there exists $y\in G$ with $yxy^{-1}=x^{1+kp^a}$ for some $0<k<p$. We observe that $y\in\N_G(\langle x\rangle N)=\N_G(P)$. Replacing $y$ by its $p$-part, we get $y\in P$. Now $x^{-kp^a}=[x,y]\in P'$ and $\exp(P/P')=p^a$ as desired.
\end{proof}

If $G$ is $p$-solvable in the situation of \autoref{corpsolv} (recall that every $p$-solvable group is $p$-constrained), then $\pcore_p(G)$ has a complement $K$ in $\pcore_{pp'}(G)$ by the Schur--Zassenhaus theorem. Using the Frattini argument, it is easy to show that $\N_G(K)$ is a complement of $N$ in $G$. In this situation, $G$ is a primitive permutation group on $N$ of affine type.

On the other hand, every non-abelian simple group $S$ gives rise to a non-split extension $G=N.S$ where $N=\Phi(G)$ is elementary abelian without complement (see \cite[Theorem~B.11.8]{DoerkHawkes}). Garrison~\cite{Garrison} has exhibited examples to show that $X(G)$ does not determine whether $G$ splits over $N$. For instance, $\mathtt{PerfectGroup}(7500,1)\cong C_5^3\rtimes A_5$ and $\mathtt{PerfectGroup}(7500,2)\cong C_5^3.A_5$ in GAP~\cite{GAP48} have the same character table and the Sylow $5$-subgroup is $2$-generated in both cases.

Now assume that $N=\F^*(G)$ in the situation of \autoref{2gen} is non-abelian.
If $N\cap P$ is abelian, then $N$ has a complement in $PN$ by \cite[Satz~IV.3.8]{Huppert}. In this case $PN$ is a twisted wreath product. The non-split extension $M_{10}=A_6.C_2$ with $P=SD_{16}$, a semidihedral group, shows that this is not always the case. Even when $N$ is not simple, $P\cap N$ is not always abelian (as in \cite[Theorem~3.1]{NS22}). One example is
\[G=\PSL(2,7)^2\rtimes\langle x\rangle\cong\PSL(2,7)^2\rtimes C_4\le\PGL(2,7)\wr C_2\]
where $x^2$ acts as a diagonal automorphism on both factors $\PSL(2,7)$ simultaneously. Here $P=D_8^2\rtimes C_4$ is $2$-generated.
Nevertheless, we provide the following reduction theorem.

\begin{Prop}\label{lemred}
Let $G$ be a finite group with Sylow $p$-subgroup $P$ such that $\pcore_{p'}(G)=1$ and $N=\F^*(G)$ is the unique minimal normal subgroup of $G$. Suppose that $N$ is non-abelian and $PN/N$ is cyclic.
Let $S$ be a simple component of $N$. Assume that $|G:\N_G(S)|$ is a $p$-power. Then the following holds:
\begin{enumerate}[(i)]
\item\label{lema} $G=\N_G(S)P$.
\item $\tilde{P}:=\N_P(S)\C_G(S)/\C_G(S)$ is a Sylow $p$-subgroup of the almost simple group $\N_G(S)/\C_G(S)$ with socle $\tilde{S}:=S\C_G(S)/\C_G(S)\cong S$. Moreover, $\tilde{P}\tilde{S}/\tilde{S}$ is cyclic.
\item $|P:\Phi(P)|\le p^2$ if and only if $|\tilde{P}:\Phi(\tilde{P})|\le p^2$.
\item $S$ and $|\tilde{P}|$ are determined by $X(G)$.
\end{enumerate}
\end{Prop}
\begin{proof}\hfill
\begin{enumerate}[(i)]
\item Since $|G:\N_G(S)|$ is a $p$-power, $|\N_G(S)P|=|\N_G(S):\N_P(S)||P|=|G|$ and $G=\N_G(S)P$.

\item By \eqref{lema}, $\N_P(S)$ is a Sylow $p$-subgroup of $\N_G(S)$. Hence, $\tilde{P}$ is a Sylow $p$-subgroup of $\N_G(S)/\C_G(S)$.
Let $Q:=N\cap P\unlhd P$. Then $P/Q\cong PN/N$ is cyclic by hypothesis. Let $x\in P$ such that $P=\langle x\rangle Q$. Then $\tilde{P}\tilde{S}/\tilde{S}\cong \N_P(S)S\C_G(S)/S\C_G(S)\le \langle x\rangle S\C_G(S)/S\C_G(S)$ is cyclic.

\item\label{Ptilde} If $P\le N\le\N_G(S)$, then $S\unlhd G$ and $N=S$. Here, $P\cong\tilde{P}$, so we are done. Now assume $PN/N\ne 1$.
Since $\pcore^p(PN)=N$, there exists $x\in P$ such that $P=\langle x\rangle Q$ and $\langle x\rangle\cap Q\le P'$ (see \cite[Satz~3.3]{Brandis}). \autoref{lemPhi} yields $|P:\Phi(P)|=p^2$ if and only if $|\C_{Q/\Phi(Q)}(x)|=p$.

By \eqref{lema}, we may write $N=T_1\times\ldots\times T_{p^a}$ such that $T_i=x^{i-1}Sx^{1-i}$ for $i=1,\ldots,p^a$. Let $Q_i:=T_i\cap P\le Q$. Then $\tilde{Q}:=Q_1\C_G(S)/\C_G(S)\cong Q_1$ is a normal subgroup of $\tilde{P}$.
Since $\N_P(S)=\langle x^{p^a}\rangle Q$, we have $\tilde{P}=\langle\tilde{x}\rangle\tilde{Q}$ where $\tilde{x}:=x^{p^a}\C_G(S)$.
It is easy to see that the map
\[\C_{Q_1/\Phi(Q_1)}(x^{p^a})\to \C_{Q/\Phi(Q)}(x),\quad y\Phi(Q_1)\mapsto \prod_{i=0}^{p^a-1}x^iyx^{-i}\Phi(Q)\]
is an isomorphism. In particular, $|\C_{Q/\Phi(Q)}(x)|=|\C_{Q_1/\Phi(Q_1)}(x^{p^a})|$.
Assume for the moment that $x^{p^a}\in Q$. Then $\tilde{P}=\tilde{Q}\le\tilde{S}$ and $|\C_{Q_1/\Phi(Q_1)}(x^{p^a})|=|Q_1/\Phi(Q_1)|=|\tilde{P}/\Phi(\tilde{P})|$. In this case, $|P:\Phi(P)|=p^2$ if and only if $\tilde{P}$ is cyclic, i.\,e. $|\tilde{P}:\Phi(\tilde{P})|=p$. Now let $x^{p^a}\notin Q$. By way of contradiction, suppose that $x^{p^a}\in Q_1\C_G(S)$. Then there exists $y\in Q_1$ such that $x^{p^a}y\in\C_G(S)$. Now also
\[z:=x^{p^a}\prod_{i=0}^{p^a-1}x^iyx^{-i}\in\C_G(S).\]
Since $z$ is centralized by $x$, it follows that $z\in x^i\C_G(S)x^{-i}=\C_G(T_i)$ for $i=1,\ldots,p^a$. Hence, $z\in\C_G(N)=1$ and $x^{p^a}\in Q$, a contradiction. Thus, $\tilde{Q}<\tilde{P}$ and
\[\tilde{Q}\cap\langle\tilde{x}\rangle=(Q\cap\langle x^{p^a}\rangle)\C_G(S)/\C_G(S)\le P'\C_G(S)/\C_G(S)=\tilde{P}'.\]
\autoref{lemPhi} shows that $|\tilde{P}:\Phi(\tilde{P})|=p^2$ if and only if
\[|\C_{Q_1/\Phi(Q_1)}(x^{p^a})|=|\C_{\tilde{Q}/\Phi(\tilde{Q})}(\tilde{x})|=p.\]
Now the claim follows.

\item The isomorphism types of $N$ and $S$ are determined by $X(G)$ according to \cite[Theorem~4.1]{NS22}.
We obtain $|\N_P(S)|$ from $|N|=|S|^{|P:\N_P(S)|}$. Arguing as in \eqref{Ptilde}, shows that $\C_P(S)=\C_Q(S)=Q_2\ldots Q_{p^a}$. Hence, $|\C_P(S)|=|S|_p^{p^a-1}$ is computable from $X(G)$. The claim follows from $\tilde{P}\cong\N_P(S)/\C_P(S)$.\qedhere
\end{enumerate}
\end{proof}

To decide whether $|P:\Phi(P)|=p^2$ holds, it suffices to obtain the structure of $\tilde{P}$ with the notation from \autoref{lemred}.
If $p\ge 5$ and $S$ is neither a linear nor a unitary group, then $\Out(S)$ has a cyclic Sylow $p$-subgroup by \cite[Table~5]{Atlas}. In this case the isomorphism type of $\tilde{P}$ is uniquely determined by $X(G)$ and the problem is solved.
On the other hand, the proof of \cite[Lemma~5.1]{NS22} shows that for linear and unitary groups $S$ the condition $|P:\Phi(P)|=p^2$ is not determined by $|\tilde{P}|$ alone. It remains a challenge to settle these cases (and $p=3$ with $S=D_4(q)$, $E_6(q)$ and $^2E_6(q)$).

\section{\texorpdfstring{$p$}{p}-groups of maximal class}\label{secmax}

We start by introducing some terminology of (saturated, non-exotic) fusion systems. Let $P$ be a Sylow $p$-subgroup of $G$ as before.
The \emph{fusion system} $\FS=\FS_P(G)$ of $G$ on $P$ is a category whose objects are the subgroups of $P$ and the morphisms of $\FS$ have the form $f:S\to T$, $x\mapsto gxg^{-1}$ where $S,T\le P$ and $g\in G$. Then $\Aut_\FS(S)\cong\N_G(S)/\C_G(S)$ and $\Out_\FS(S)\cong\N_G(S)/S\C_G(S)$. Elements $x,y\in P$ (or subsets $S,T\subseteq P$) are called $\FS$-\emph{conjugate} if there exists a morphism $f$ such that $f(x)=y$ (or $f(S)=T$). A subgroup $S\le P$ is called
\begin{itemize}
\item \emph{fully normalized}, if $|\N_P(T)|\le|\N_P(S)|$ for all $\FS$-conjugates $T$ of $S$.
\item \emph{centric}, if $\C_P(T)=\Z(T)$ for all $\FS$-conjugates $T$ of $S$.
\item \emph{radical}, if $\pcore_p(\Aut_\FS(S))=\Inn(S)$ (equivalently, $\pcore_p(\Out_\FS(S))=1$).
\item \emph{essential}, if $S$ is fully normalized, centric and $\Out_\FS(S)$ contains a strongly $p$-embedded subgroup (see \cite[Definition A.6]{AKO}). For our purpose, it is enough to know that $S$ is radical in this case.
\end{itemize}

By Alperin's fusion theorem, every morphism in $\FS$ is a composition of restrictions of morphisms $f\in\Aut_\FS(S)$ where $S=P$ or $S$ is essential (see \cite[Theorem~I.3.5]{AKO}).
Note that $\Aut_\FS(P)$ permutes the essential subgroups by conjugation. Hence, if $Q\le P$ does not lie in any essential subgroup, then $Q$ is fully normalized. In this case, $\N_P(Q)$ is a Sylow $p$-subgroup of $\N_G(Q)$ (see \cite[Lemma~I.1.2]{AKO}). Consequently, $\C_P(Q)=\N_P(Q)\cap\C_G(P)$ is a Sylow $p$-subgroup of $\C_G(P)$.

We call $\FS$ \emph{controlled} if $\N_G(P)$ controls the fusion in $P$ with respect to $G$, i.\,e. every morphism $S\to T$ has the form $x\mapsto gxg^{-1}$ for some $g\in\N_G(P)$. Abstractly, this means that there are no essential subgroups and $\FS=\FS_P(P\rtimes A)$ for some Schur--Zassenhaus complement $A$ of $\Inn(P)$ in $\Aut_\FS(P)$.
More generally, $\FS$ is called \emph{constrained} if there exists $Q\unlhd P$ such that $\C_P(Q)=\Z(Q)$ and $\N_G(Q)$ controls the fusion in $P$. By the model theorem (see \cite[Theorem~I.4.9]{AKO}), a constrained fusion system is realized by a unique group $G$ such that $\C_G(\pcore_p(G))\le\pcore_p(G)$ (note that $G$ is $p$-constrained).
The largest subgroup $Q\unlhd P$ such that $\N_G(Q)$ controls the fusion in $P$ is denoted by $\pcore_p(\FS)$. Note that $\pcore_p(G)\le\pcore_p(\FS)$.

It is well-known that a $p'$-automorphism of $Q\le P$ acts non-trivially on $Q/\Phi(Q)$. If $Q$ is radical, it follows that $\Out_\FS(Q)$ acts faithfully on $Q/\Phi(Q)$. Now assume that there exists a series of characteristic subgroups $\Phi(Q)=Q_0<\ldots< Q_n=Q$ of $Q$. Then $\Out_\FS(Q)$ acts faithfully on $Q_n/Q_{n-1}\times\ldots\times Q_1/Q_0$ by \cite[5.3.2]{Gorenstein}.
This argument will often be applied in the following to exclude same candidates of essential subgroups.

We say that a $p$-group $P$ of order $p^n$ has \emph{maximal class} if the nilpotency class is $n-1$. This may include the case $|P|=p^2$. The $2$-groups of maximal class are the dihedral groups (including $C_2^2$), the semidihedral groups, the (generalized) quaternion groups and $C_4$ (see \cite[Satz~III.11.9]{Huppert}).
Now assume that $n\ge 4$ and $p>2$ to avoid some degenerate cases. Let $\K_2(P)=P'$ and $\K_{i+1}(P)=[P,\K_i(P)]$ for $i\ge 2$. Let $\Z_0(P):=1$ and $\Z_{i+1}(P/\Z_i(P)):=\Z(P/\Z_i(P))$ for $i\ge 0$. Then $\K_i(P)=\Z_{n-i}(P)$ is the only normal subgroup of $P$ of index $p^i$ by \cite[Hilfssatz~III.14.2]{Huppert}. It is easy to see that the characteristic subgroups $P_1:=\C_P(\K_2(P)/\K_4(P))$ and $P_2:=\C_P(\Z_2(P))$ are maximal in $P$.

\begin{Lem}\label{selfcent}
Let $P$ be a $p$-group with a non-abelian subgroup $Q\le P$ of order $p^3$ and exponent $p$. If $\C_P(Q)=\Z(Q)$, then $\Z_2(P)\le Q$.
\end{Lem}
\begin{proof}
Since $\Z(P)\le\C_P(Q)$, we have $Z:=\Z(P)=\Z(Q)\cong C_p$. Let $xZ\in\C_{P/Z}(Q/Z)$. Then $x\in\N_P(Q)$.
By \cite{Winter}, $\N_P(Q)/Q\le\Out(Q)\cong\GL(2,p)$. As mentioned above, the kernel of the action of $\Aut(Q)$ on $Q/Z$ is a $p$-group. Since $\pcore_p(\GL(2,p))=1$, we obtain $x\in Q$. Hence, $\Z_2(P)/Z=\Z(P/Z)\le\C_{P/Z}(Q/Z)=Q/Z$ and $\Z_2(P)\le Q$.
\end{proof}

\begin{Lem}\label{normal}
Let $G$ be a finite group with Sylow $p$-subgroup $P$ of maximal class. Let $N\unlhd G$ such that $p^2\le|N|_p<|P|$.
Then there exists $x\in P$ such that $|\C_G(x)|_p=p^2$.
\end{Lem}
\begin{proof}
By hypothesis, $|P|\ge p|N|_p\ge p^3$. In particular, $\Z(P)$ is the unique normal subgroup of order $p$ of $P$. Since $M:=P\cap N\unlhd P$, we have $\Z(P)\le N$.
If $|P|=p^3$, every element $x\in P\setminus N$ cannot be conjugate to an element of $\Z(P)\le N$. Hence, $|\C_G(x)|_p=p^2$. Now assume that $|P|\ge p^4$. If $p=2$, $P$ is a dihedral, semidihedral or quaternion group and we choose $x\in P$ outside the cyclic maximal subgroup of $P$. For $p>2$, let $x\in P\setminus(P_1\cup P_2)$. By \cite[Hilfssatz~III.14.13]{Huppert}, we have $|\C_P(x)|=p^2$.
Since $|P|\ge p^4$, $\Z_2(P)$ is the unique normal subgroup of order $p^2$ in $P$. In particular, $\Z_2(P)\le M$ since $|M|\ge p^2$.
If $p=2$, we may assume that $x\notin M$. For $p>2$, we have $P_1\cup P_2\cup M\subsetneq P$. Again we may choose $x\notin M$.

Let $\FS$ be the fusion system of $G$ on $P$. If $x$ is not contained in any essential subgroup, then $\langle x\rangle$ is fully normalized as explained above. It follows that $|\C_G(x)|_p=|\C_P(x)|=p^2$ and we are done. Now let $Q<P$ be essential containing $x$.
By \cite[Theorem~D]{GrazianParker}, $Q$ is a so-called pearl, i.\,e. $Q$ is elementary abelian of order $p^2$ or non-abelian of order $p^3$ and exponent $p$ (or $Q=Q_8$ if $p=2$, see \cite[Lemma~6.1]{GrazianParker}). As an essential subgroup, $Q$ is centric and $\C_P(Q)=\Z(Q)$. Assume first that $|Q|=p^2$. Then
\[Z:=\Z(P)=M\cap Q=N\cap Q\unlhd\N_G(Q).\]
Since $Q$ is radical, $\Out_\FS(Q)\cong\N_G(Q)/Q$ acts faithfully on $Z\times Q/Z\cong C_p^2$. But then $\Out_\FS(Q)$ would be a $p'$-group in contradiction to $Q<\N_P(Q)$. Next let $|Q|=p^3$. Here, \autoref{selfcent} shows that $\Z_2(P)=M\cap Q=N\cap Q\unlhd\N_G(Q)$. Then $\Out_\FS(Q)$ acts faithfully on $\Z_2(P)/Z\times Q/\Z_2(P)\cong C_p^2$ and we derive another contradiction.
\end{proof}

\begin{ThmB}
Let $G$ be a finite group with a Sylow $p$-subgroup $P$ of maximal class. Suppose that $\pcore_{p'}(G)=1$ and $\pcore^{p'}(G)=G$. Then one of the following holds:
\begin{enumerate}[(i)]
\item\label{cent} There exists $x\in P$ such that $|\C_G(x)|_p=p^2$.
\item\label{quasi} $G$ is quasisimple and $|\Z(G)|\le p$.
\end{enumerate}
\end{ThmB}
\begin{proof}
We may assume that $G$ is not simple and $|P|\ge p^3$. Let $N<G$ be a maximal normal subgroup. Then $1<|N|_p<|P|$ as $\pcore_{p'}(G)=1$ and $\pcore^{p'}(G)=G$. If $|N|_p\ge p^2$, then the claim follows from \autoref{normal}. Hence, let $|N|_p=p$.
Then $P\cap N\unlhd P$ has index $p^s\ge p^2$ and therefore $P\cap N=\K_s(P)\le P'$. By Tate's theorem (see \cite[Satz~IV.4.7]{Huppert}), $N$ has a normal $p$-complement. Since $\pcore_{p'}(G)=1$, this forces $|N|=p$. Since $|G:\C_G(N)|$ divides $p-1$, we further have $N\le\Z(G)$. Since $G/N$ is simple, $G$ is quasisimple with $|\Z(G)|\le p$.
\end{proof}

If Case~\eqref{quasi} in Theorem~B applies with $|\Z(G)|=p$ and \eqref{cent} fails, then Robinson's ordinary weight conjecture predicts the existence of an irreducible character $\chi$ in the principal $p$-block such that $p^2\chi(1)_p=|G|_p$ (see \cite[Lemma~4.7]{Robinsonmetac}). Conversely, such a character can only appear when $P$ has maximal class.
Examples are $\SL(2,9)$ for $p=2$, $\SL(3,19)$ for $p=3$ and $\SL(p,q)$ for $p\ge 5$ where $q-1$ is divisible by $p$ just once.
Our proof of Theorem~A does however not rely on any conjecture.

\begin{ThmA}
The character table of a finite group $G$ determines whether $G$ has Sylow $p$-subgroups of maximal class.
\end{ThmA}
\begin{proof}
Let $P$ be a Sylow $p$-subgroup of $G$. We may assume that $\pcore_{p'}(G)=1$ and $|P|\ge p^3$.
Let $K:=\pcore^{p'}(G)$.
The character table detects elements $x\in P$ such that $|\C_G(x)|_p=|\C_K(x)|_p=p^2$. In this case $|\C_P(x)|=p^2$ and $P$ has maximal class by \cite[Satz~III.14.23]{Huppert}. Hence, by Theorem~B we may assume that $K$ is quasisimple with $|\Z(K)|\le p$. Note that the character table of $G$ determines the isomorphism type of the simple chief factor $K/\Z(K)$ (see \cite[Theorem~4.1]{NS22}). In this way we confirm that the Sylow $p$-subgroup $P/\Z(K)$ of $K/\Z(K)$ has maximal class. If $\Z(K)=1$, then we are done. Otherwise, $P$ has maximal class if and only if $\Z(K)=\Z(P)$. This happens if and only if $|\C_G(x)|_p<|P|$ for all $x\in P\setminus\Z(K)$.
\end{proof}

\section{Minimal non-abelian Sylow subgroups}\label{secmna}

The following elementary lemma underlines the importance of minimal non-abelian groups. For elements $x,y,z$ of a group we use the commutator convention $[x,y,z]:=[x,[y,z]]$.

\begin{Lem}\label{lemmna}
For a $p$-group $P$ the following assertions are equivalent:
\begin{enumerate}[(1)]
\item $P$ is minimal non-abelian.
\item $|P:\Phi(P)|=|P:\Z(P)|=p^2$.
\item $|P:\Phi(P)|=p^2$ and $|P'|=p$.
\end{enumerate}
\end{Lem}
\begin{proof}
\begin{description}
\item[$(1)\Rightarrow(2)$:] Since $P$ is non-abelian, there exist non-commuting elements $x,y\in P$. Since $\langle x,y\rangle$ is non-abelian, we have $P=\langle x,y\rangle$. By Burnside's basis theorem, $|P:\Phi(P)|=p^2$. Choose distinct maximal subgroups $S,T<P$. Since $S$ and $T$ are abelian and $P=ST$, it follows that $\Phi(P)=S\cap T\subseteq\Z(P)$. It is well-known that $P/\Z(P)$ cannot be a non-trivial cyclic group. In particular, $|P:\Z(P)|\ge p^2$ and $\Phi(P)=\Z(P)$.

\item[$(2)\Rightarrow(3)$:] Let $\Z(P)<S<P$. Since $S/\Z(P)$ is cyclic and $\Z(P)\le\Z(S)$, we obtain that $S$ is abelian. Pick $x\in P\setminus S$. Then \autoref{lemabel} yields that $|P'|=|S:\Z(P)|=p$.

\item[$(3)\Rightarrow(1)$:] Obviously, $P$ is non-abelian since $P'\ne 1$.
For $g,x\in P$ we have $gxg^{-1}=[g,x]x\in P'x$. Thus, every conjugacy class lies in a coset of $P'$. The hypothesis $|P'|=p$ implies $|P:\C_P(x)|\le p$ for every $x\in P$. Since $\Phi(P)$ is the intersection of the maximal subgroups of $P$, we deduce $\Phi(P)\le\bigcap_{x\in P}\C_P(x)=\Z(P)$. Now for every maximal subgroup $S<P$, we see that $S/\Z(S)$ is cyclic and $S$ must be abelian. In total, $P$ is minimal non-abelian. \qedhere
\end{description}
\end{proof}

The non-nilpotent, minimal non-abelian groups were classified by Miller--Moreno~\cite{MillerMoreno}. The nilpotent ones are $p$-groups and have been determined by Rédei~\cite{Redei}. For the convenience of the reader we give a proof.

\begin{Lem}[Rédei]\label{redei}
Every minimal non-abelian $p$-group belongs to one of the following classes:
\begin{enumerate}[(i)]
\item $\Gamma(a,b):=\langle x,y\mid x^{p^a}=y^{p^b}=1,\ yxy^{-1}=x^{1+p^{a-1}}\rangle$ a metacyclic group where $a\ge 2$ and $b\ge 1$.
\item $\Delta(a,b):=\langle x,y\mid x^{p^a}=y^{p^b}=[x,y]^p=[x,x,y]=[y,x,y]=1\rangle$ where $a\ge b\ge 1$.
\item $Q_8$.
\end{enumerate}
\end{Lem}
\begin{proof}
Let $P$ be minimal non-abelian. By \autoref{lemmna}, there exist $x,y\in P$ such that $P/P'=\langle xP'\rangle\times\langle yP'\rangle\cong C_{p^a}\times C_{p^b}$. Since $|P'|=p$, we have $P'=\langle z\rangle$ where $z:=[x,y]$.
Note that $P'\le\Phi(P)=\Z(P)$ and $[x,z]=[y,z]=1$. We distinguish three cases:

\textbf{Case~(1):} $x^{p^a}=y^{p^b}=1$.\\
Here $P$ fulfills the same relations as $\Delta(a,b)$, so it must be a quotient of the latter group. Moreover, every element of $P$ can be written uniquely in the form $x^iy^jz^k$ with $1\le i\le p^a$, $1\le j\le p^b$ and $1\le k\le p$. Consequently, $|P|=p^{a+b+1}$. For the same reason we have $|\Delta(a,b)|\le p^{a+b+1}$. Therefore, $P\cong\Delta(a,b)$.

\textbf{Case~(2):} Either $x^{p^a}=1$ or $y^{p^b}=1$.\\
Without loss of generality, let $x^{p^a}\ne 1$ and $y^{p^b}=1$. Then $P'\le\langle x\rangle\unlhd P$ and $yxy^{-1}=x^k$ for some $k\in\ZZ$. Since $\langle x^p,y\rangle<P$ is abelian, $x^p=yx^py^{-1}=x^{kp}$ and $p\equiv kp\pmod{p^{a+1}}$ as $|\langle x\rangle|=p^{a+1}$. Hence, we may assume that $k=1+p^al$ for some $0<l<p$. Let $0<l'<p$ such that $ll'\equiv 1\pmod{p}$. Then $y^{l'}xy^{-l'}=x^{(1+p^al)^{l'}}=x^{1+p^a}$. Thus, after replacing $y$ by $y^{l'}$, we obtain $yxy^{-1}=x^{1+p^a}$. Now $P$ satisfies the relations of $\Gamma(a+1,b)$. It is clear that these groups have the same order, so $P\cong\Gamma(a+1,b)$.

\textbf{Case~(3):} $x^{p^a}\ne 1\ne y^{p^b}$.\\
Without loss of generality, let $a\ge b$. Let $x^{p^a}=z^i$ and $y^{p^b}=z^j$ where $0<i,j<p$. Then $(x^j)^{p^a}=z^{ij}$, $(y^i)^{p^b}=z^{ij}$ and $[x^j,y^i]=z^{ij}$ by \cite[Hilfssatz~III.1.3]{Huppert} (using $z\in\Z(P)$). Hence, replacing $x$ by $x^j$ and $y$ by $y^i$, we may assume that $x^{p^a}=z=y^{p^b}$.
Again by \cite[Hilfssatz~III.1.3]{Huppert},
\[(x^{p^{a-b}}y^{-1})^{p^b}=x^{p^a}y^{-p^b}[y^{-1},x^{p^{a-b}}]^{\binom{p^b}{2}}=z^{p^{a-b}\binom{p^b}{2}}=1\]
unless $p^b=p^a=2$. In this exceptional case, $P\cong Q_8$. Otherwise, we replace $y$ by $x^{p^{a-b}}y^{-1}$.
Afterwards we still have $P/P'=\langle xP'\rangle\times\langle yP'\rangle$, but now $y^{p^b}=1$. Thus, we are in Case~(2).
\end{proof}

The metacyclic groups $\Gamma(a,b)$ can of course be constructed as semidirect products, while the groups $\Delta(a,b)$ can be constructed as subgroups of $\Gamma(a,b)\times C_{p^a}$.
For $p=2$, note that $\Gamma(2,1)\cong D_8\cong \Delta(1,1)$. Apart from that, the groups in \autoref{redei} are pairwise non-isomorphic (for different parameters $a,b$).

We digress slightly to present a counterexample to a related question. Since for $p$-groups $P$ in general we have $\Phi(P)=P'\mho(P)$ where $\mho(P)=\langle x^p:x\in P\rangle$, one might wonder if $X(G)$ determines the property $|P:\mho(P)|=p^2$. For $p=2$, it is well-known that $\mho(P)=\Phi(P)$, so the answer is yes in this case. For $p>2$, $|P:\mho(P)|=p^2$ holds if and only if $P$ is metacyclic (see \cite[Satz~III.11.4]{Huppert}). The following example shows that this is not even determined by $X(P)$.

\begin{Prop}
For $a\ge 2$ and all primes $p$ the groups $\Gamma(2,a)$ and $\Delta(a,1)$ have the same character table.
\end{Prop}
\begin{proof}
We denote the generators of $P:=\Gamma(2,a)$ by $x,y$ and those of $\tilde{P}:=\Delta(a,1)$ by $\tilde{x},\tilde{y}$ as in \autoref{redei}. Additionally, let $\tilde{z}:=[\tilde{x},\tilde{y}]$. We consider the maximal subgroups $Q:=\langle x^p,y\rangle\le P$ and $\tilde{Q}:=\langle\tilde{x},\tilde{z}\rangle\le\tilde{P}$. Since $xyx^{-1}=x^{-p}y$ and $\tilde{y}\tilde{x}\tilde{y}^{-1}=\tilde{z}^{-1}\tilde{y}$, the map
\[Q\to\tilde{Q},\qquad x^p\mapsto z,\ y\mapsto\tilde{x}\]
 is an isomorphism compatible with the action of $P$ and $\tilde{P}$. The irreducible characters of $P$ of degree $p$ are induced from linear characters of $Q$, which are not $P$-invariant. Since these characters vanish outside $Q$, they correspond naturally to irreducible characters of $\tilde{P}$. On the other hand, the linear characters of $P$ are extensions of characters of $Q$ with $x^p$ in their kernel. For $\lambda\in\Irr(Q/P')$ the extensions $\hat{\lambda}$ are determined by $\hat{\lambda}(x)=\zeta$ where $\zeta$ is a $p$-th root of unity. Similarly, for $\lambda\in\Irr(\tilde{Q}/\tilde{P}')$ the extensions are determined by $\hat{\lambda}(\tilde{y})=\zeta$.
Therefore, the bijection $P\to\tilde{P}$, $x^{i+jp}y^k\mapsto\tilde{x}^k\tilde{y}^i\tilde{z}^j$ where $0\le i,j<p$ and $0\le k<p^a$ induces the equality of the matrices $X(P)$ and $X(\tilde{P})$.
\end{proof}

The second author has investigated fusion systems on minimal non-abelian $2$-groups in order to classify blocks with such defect groups (see e.\,g. \cite{Sambalemna3}). We now determine the fusion systems for odd primes too (partial results were obtained in \cite{GaoControl}). It turns out that they all come from finite groups unless $|P|=7^3$.
We make use of the Frobenius group $M_9\cong\PSU(3,2)\cong C_3^2\rtimes Q_8$ with $\Out(M_9)\cong S_3$.

\begin{Thm}\label{thmfusion}
Let $\FS$ be a saturated fusion system on a minimal non-abelian $p$-group $P$. Then one of the following holds
\begin{enumerate}[(i)]
\item $P\in\{D_8,Q_8\}$ and $\FS=\FS_P(G)$ where $G\in\{P,S_4,\GL(3,2),\SL(2,3)\}$.
\item\label{exp} $|P|=p^3$, $\exp(P)=p>2$ and the possibilities for $\FS$ are given in \cite{ExtraspecialExpp}.
\item\label{meta} $P\cong \Gamma(a,b)$, $a\ge 2$, $b\ge 1$ and $\FS=\FS_P(C_{p^a}\rtimes C_{p^bd})$ for some $d\mid p-1$.
\item\label{contro} $P\cong \Delta(a,b)$, $a>b$ and $\FS=\FS_P(P\rtimes Q)$ where $Q\le C_{p-1}^2$.
\item\label{contro2} $P\cong \Delta(a,a)$, $a\ge 2$ and $\FS=\FS_P(P\rtimes Q)$ for some $p'$-group $Q\le\GL(2,p)$.
\item\label{const2} $p=2$, $P\cong\Delta(a,1)$, $a\ge 2$ and $\FS=\FS_P(A_4\rtimes C_{2^a})$ where $C_{2^a}$ acts as a transposition in $\Aut(A_4)=S_4$.
\item\label{const3} $p=3$, $P\cong\Delta(a,1)$, $a\ge 2$ and $\FS=\FS_P(G)$ where $G\in\{M_9\rtimes C_{3^a},M_9\rtimes D_{2\cdot3^a}\}$. Here the image of $C_{3^a}$ and $D_{2\cdot 3^a}$ in $\Out(M_9)$ is $C_3$ and $S_3$ respectively.
\end{enumerate}
\end{Thm}
\begin{proof}
The case $P\in\{D_8,Q_8\}$ is well-known and can be found in \cite[Theorem~5.3]{CravenGlesser}, for instance.
If $p=2$ and $P=\Gamma(a,b)$ with $|P|\ge 16$, then $\FS$ is trivial, i.\,e. $\FS=\FS_P(P)$ by \cite[Theorem~3.7]{CravenGlesser}. Then \eqref{meta} holds. Now suppose that $p>2$ and $P=\Gamma(a,b)$. Then $\FS$ is controlled, i.\,e. $\FS=\FS_P(P\rtimes Q)$ for some $p'$-group $Q\le\Aut(P)$ by \cite{Stancu} (see also \cite[Theorem~3.10]{CravenGlesser}).
By \cite[Lemma~2.4]{Sasaki}, $\Aut(P)=A\rtimes\langle\sigma\rangle$ where $A$ is a $p$-group, $|\langle\sigma\rangle|=p-1$, $\sigma(x)\in\langle x\rangle$ and $\sigma(y)=y$. Hence, $Q$ is conjugate to a subgroup of $\langle\sigma\rangle$. After renaming the generators of $P$, we may assume that $Q\le\langle\sigma\rangle$. Now \eqref{meta} holds.

Next assume that $P\cong \Delta(a,b)$ for some $a\ge b\ge 1$. If $a=1$ and $p>2$, then $|P|=p^3$ and $\exp(P)=p$, so \eqref{exp} holds. Hence, let $a\ge 2$. Set $z:=[x,y]\in P$.
Since the $p'$-group $\Out_\FS(P)$ acts faithfully on $P/\Phi(P)\cong C_p^2$, we have $\Out_\FS(P)\le\GL(2,p)$. If $a>b$, then $\Out_\FS(P)$ acts on $P/\Omega_{a-1}(P)\times\Omega_{a-1}(P)/\Phi(P)$ where $\Omega_{a-1}(P)=\langle g\in P:g^{p^{a-1}}=1\rangle=\langle x^p,y,z\rangle$. In this case $\Out_{\FS}(P)\le C_{p-1}^2$. If $\FS$ is controlled, then we are in Case~\eqref{contro} or \eqref{contro2}.
Hence, we may assume that $\FS$ is not controlled. Then there exists an essential subgroup $Q\le P$. Since $Q$ is centric and $\Phi(P)=\Z(P)\le\C_P(Q)\le Q$, $Q$ is a maximal subgroup. Those are given by
\begin{align*}
\langle xy^i,y^p,z\rangle&\cong C_{p^a}\times C_{p^{b-1}}\times C_p&(i=0,\ldots,p-1),\\
\langle x^p,y,z\rangle&\cong C_{p^{a-1}}\times C_{p^b}\times C_p.
\end{align*}
By \cite[Theorem~5.2.4]{Gorenstein}, $A:=\Aut_{\FS}(Q)$ acts faithfully on $\Omega(Q)=\{g\in Q:g^p=1\}$.
Since $P/Q\le A$, this implies $\Omega(Q)\nsubseteq\Z(P)$ and $Q=\langle x^p,y,z\rangle$ with $b=1$.
Now $Q$ is the only maximal subgroup of $P$ isomorphic to $C_{p^{a-1}}\times C_p^2$. In particular, $Q$ is characteristic in $P$. By Alperin's fusion theorem, $\FS$ is constrained with $\pcore_p(\FS)=Q$. By the model theorem, there exists a unique $p$-constrained group $H$ with $P\in\Syl_p(H)$, $\pcore_{p'}(H)=1$ and $\FS=\FS_P(H)$. We will construct $H$ in the following.

By \cite[Lemma~1.11]{Oliverindexp}, there exists an $A$-invariant decomposition $Q=Q_1\times Q_2$ with $Q_1\cong C_p^2$ and $Q_2\cong C_{p^{a-1}}$. Moreover, $\pcore^{p'}(A)\cong\SL(2,p)$ acts faithfully on $Q_1$ and trivially on $Q_2$. Since $P/Q\le\pcore^{p'}(A)$, it follows that $Q_2\le\Z(P)=\langle x^p,z\rangle$. Moreover, $xyx^{-1}=yz$ implies $z\in Q_1$.
Let $\alpha\in A$ be a $p'$-automorphism acting trivially on $Q_1$.
Then $\alpha$ commutes with the action of $P/Q$. Since $Q$ is receptive (see \cite[Definition~I.2.2]{AKO}), $\alpha$ extends to an automorphism of $P$. Suppose that $\alpha\ne 1$. Since $Q_2\le\Z(P)=\Phi(P)$, $\alpha$ must act non-trivially on $P/Q_2$. Note that $P/Q_2$ is non-abelian of order $p^3$ as $z\in Q_1$. An analysis of $\Aut(P/Q_2)$ reveals that $\alpha$ cannot act trivially on $Q/Q_2\cong Q_1$. Hence, $\alpha=1$ and $A$ acts faithfully on $Q_1$. In particular, $A\le\GL(2,p)$.
If $p=2$, then $A\cong\SL(2,2)=\GL(2,2)\cong S_3$. It is easy to see that \eqref{const2} holds here.
If $p=3$, then $\SL(2,3)\cong Q_8\rtimes C_3$, $\GL(2,3)\cong Q_8\rtimes S_3$ and \eqref{const3} is satisfied. Thus, let $p\ge 5$. Then the Sylow normalizer in $\SL(2,p)$ acts non-trivially on a Sylow $p$-subgroup of $\SL(2,p)$. Hence, there exists $\alpha\in \pcore^{p'}(A)$ acting non-trivially $P/Q$.
But then $\alpha$ acts non-trivially on $\langle x^p\rangle Q_1/Q_1=Q/Q_1\cong Q_2$. This contradicts \cite[Lemma~1.11]{Oliverindexp}.
\end{proof}

The groups $A_4\rtimes C_4$, $M_9\rtimes C_9$ and $M_9\rtimes D_{18}$ can be constructed in GAP~\cite{GAP48} as $\mathtt{SmallGroup}(n,k)$ where $(n,k)\in\{(48,39),(648,534),(6^4,2892)\}$ respectively.

\begin{Cor}
Let $\FS$ be a fusion system on a minimal non-abelian $p$-group $P$ with $|P|\ge p^4$. Then $\FS$ is constrained. If $p\ge 5$, then $\FS$ is controlled.
\end{Cor}

We now gather some information on simple groups in order to prove Theorem~C.
As customary, let
\[\PSL^\epsilon(n,q):=\begin{cases}
\PSL(n,q)&\text{if }\epsilon=1,\\
\PSU(n,q)&\text{if }\epsilon=-1.
\end{cases}\]

The following is certainly known, but included for convenience.

\begin{Lem}\label{pslcyc}
Let $S=\PSL^\epsilon(n,q)$ with a cyclic Sylow $p$-subgroup and $n\ge 3$. Then there exists a unique integer $2\le d\le n$ such that $p$ divides $q^d-\epsilon^d$.
\end{Lem}
\begin{proof}
Since a non-abelian simple group cannot have cyclic Sylow $2$-subgroups, we have $p>2$.
If $p\mid q$, then a Sylow $p$-subgroup of $S$ is given by the set of unitriangular matrices. This subgroup is non-abelian since $n\ge 3$.
Now let $p\nmid q$. If $q\equiv\epsilon\pmod{p}$, then $S$ contains a subgroup of diagonal matrices isomorphic to $C_p^2$. Hence, let $q\not\equiv\epsilon\pmod{p}$. In the following we write $q^*:=q$ if $\epsilon=1$ and $q^*:=q^2$ if $\epsilon=-1$.
Let $x\in S$ be a generator of a Sylow $p$-subgroup of $S$. We identify $x$ with a preimage in $\GL(n,q^*)$. We may assume that $x$ has order $p^k$. Let $e$ be the order of $q^*$ modulo $p^k$. Then $x$ has an eigenvalue $\zeta\in\FF_{(q^*)^e}^\times$ of order $p^k$. 
Since $\tr(x)\in\FF_{q^*}$, the elements $\zeta^{(q^*)^i}$ for $i=0,\ldots,e-1$ are distinct eigenvalues of $x$. In particular, $e\le n$. If $\epsilon=1$, then $e\ge 2$ we can choose $d:=e$ in the statement. If $2e\le n$, we obtain $q^d\equiv 1\equiv \epsilon^d\pmod{p}$ for $d:=2e$.

Now suppose that $\epsilon=-1$ and $2e>n$. Since $x$ is a unitary matrix, we have $\bar{x}x^\TT=1$ where $\bar{x}=(x_{ij}^q)_{i,j}$ and $x^\TT$ is the transpose of $x$. It follows that $\zeta^{-q}$ is an eigenvalue of $x$. Since $n<2e$, there must be some $i$ with $\zeta^{q^{2i}}=\zeta^{-q}$. This shows that $q^{2i-1}\equiv -1\equiv\epsilon^{2i-1}\pmod{p^k}$. Since $q^{2(2i-1)}\equiv 1\pmod{p^k}$, we have $e\mid 2i-1\le 2(e-1)-1<2e$ and $e=2i-1$.
Hence, we can set $d:=e$.

For the uniqueness of $d$, we note that
\[|S|=\frac{q^{n(n-1)/2}}{\gcd(n,q-\epsilon)}\prod_{i=2}^n(q^i-\epsilon^i),\]
is not divisible by $p^{k+1}$, since $p^k=|\langle x\rangle|=|S|_p$.
\end{proof}

\begin{Lem}\label{simple3}
Let $S$ be a finite simple group with Sylow $3$-subgroup $C_3^2$ and outer automorphism of order $3$. Then $S\cong\PSL^\epsilon(3,q^f)$ where $\epsilon=\pm1$, $q$ is a prime and $(q^f-\epsilon)_3=3$. Moreover, $\Out(S)\cong C_3\rtimes (C_f\times C_2)$.
\end{Lem}
\begin{proof}
The simple groups with Sylow $3$-subgroup $C_3^2$ were classified in \cite[Proposition~1.2]{KoshitaniMiyachi}. The alternating groups and sporadic groups do not have outer automorphisms of order $3$. Now let $S$ be a classical group of dimension $d$ over $\FF_{q^f}$. Then $q^f\not\equiv\pm 1\pmod{9}$. This implies $3\nmid f$ and $S$ does not have field automorphisms of order $3$. According to \cite[Table~5]{Atlas}, there must be a diagonal automorphism of order $3$. This forces $d=3$ and $S=\PSL^\epsilon(3,q^f)$ such that $(q^f-\epsilon)_3=3$.
If $\epsilon=1$, then $\Out(S)=C_3\rtimes(C_f\times C_2)$ as desired. If $\epsilon=-1$, then there is no graph automorphism and instead we have a field automorphism of order $2f$. However, since $q^f\equiv 2,5\pmod{9}$, $f$ must be odd and $C_{2f}\cong C_f\times C_2$.
\end{proof}

\begin{ThmC}\label{main}
Let $G$ be a finite group with a minimal non-abelian Sylow $p$-subgroup $P$ and $\pcore_{p'}(G)=1$. Then one of the following holds:
\begin{enumerate}[(i)]
\item\label{p2} $p=2$, $P\in\{D_8,Q_8\}$ and $\pcore^{2'}(G)\in\{\SL(2,q),\PSL(2,q'),A_7\}$ where $q\equiv \pm 3\pmod{8}$ and $q'\equiv \pm 7\pmod{16}$.
\item\label{p3} $|P|=p^3$ and $\exp(P)=p>2$.
\item\label{psolv} $G=P\rtimes Q$ where $Q\le\GL(2,p)$.
\item\label{meta2} $p>2$, $\pcore^{p'}(G)=S\rtimes C_{p^a}$ where $S$ is a simple group of Lie type with cyclic Sylow $p$-subgroups. The image of $C_{p^a}$ in $\Out(S)$ has order $p$.

\item\label{sim2} $p=2$ and $G=\PSL(2,q^f)\rtimes C_{2^ad}$ where $q$ is a prime, $q^f\equiv\pm3\pmod{8}$ and $d\mid f$. Moreover, $C_{2^a}$ acts as a diagonal automorphism of order $2$ on $\PSL(2,q^f)$ and $C_d$ induces a field automorphism of order $d$.

\item\label{sim3} $p=3$ and $\pcore^{3'}(G)=\PSL^\epsilon(3,q^f)\rtimes C_{3^a}$ where $\epsilon=\pm1$, $q$ is a prime, $(q^f-\epsilon)_3=3$ and $G/\pcore^{3'}(G)\le C_f\times C_2$.
\end{enumerate}
\end{ThmC}
\begin{proof}
By \autoref{lemmna}, $|P:\Z(P)|=p^2$ and $G$ is described in \cite[Theorem~7.5]{NS22}. We go through the various cases in the notation used there:

In Case (A), $S=1$ since $\pcore_p(G)$ is not cyclic. Here $P=\F^*(G)\unlhd G$ and $\C_G(P)\le P$. Since $G/P$ acts faithfully on $P/\Phi(P)\cong C_p^2$, we have $G/P\le\GL(2,p)$ and \eqref{psolv} holds. Assume now that $P<G$.
In Case (B), the quasisimple group $C$ has a non-abelian Sylow $p$-subgroup of order $p^3$ which must coincide with $P$.
If $P=D_8$, then \eqref{p2} or \eqref{sim2} holds by the Gorenstein--Walter theorem (there are no field automorphisms of order $2$). If $P=Q_8$, the claim follows from the Brauer--Suzuki theorem and Walter's theorem. If $p>2$, then we must have $\exp(P)=p$, since otherwise the focal subgroup theorem and \autoref{thmfusion} lead to the contradiction $|P|=|P:P\cap G'|\ge p$. Thus, \eqref{p3} holds.
Case~(D) is impossible, since then $P$ has a non-abelian maximal subgroup.

Now consider Case (C), i.\,e. $\F^*(G)=\pcore_p(G)\times S$ has abelian Sylow $p$-subgroups, $S$ is a direct product of simple groups and $|G:\F^*(G)|_p=p$. Let $x\in P\setminus\F^*(G)$.

\textbf{Case~1:} $S=1$.\\
Since $\C_P(\pcore_p(G))=\pcore_p(G)$, we have $\C_G(\pcore_p(G))=\pcore_p(G)\times K$ where $K\le\pcore_{p'}(G)=1$. Hence, $G$ is $p$-constrained and $\FS_P(G)$ is given by \eqref{const2} or \eqref{const3} of \autoref{thmfusion}.
By the model theorem, the isomorphism type of $G$ is uniquely determined by $\FS_P(G)$. Since $\PSL(2,3)\cong A_4$ and $\PSU(3,2)\cong M_9$, we obtain \eqref{sim2} or \eqref{sim3}.

\textbf{Case~2:} $S\ne 1$ is not simple.\\
By \autoref{redei}, the maximal subgroups of $P$ are generated by at most three elements. Hence, $S$ is a direct product of two or three simple groups, say $S=T_1\times T_2$ or $T_1\times T_2\times T_3$. Then $p>2$ and the $T_i$ have cyclic Sylow $p$-subgroups. If $x$ does not normalize some $T_i$, then $p=3$ and $x$ permutes $T_1\cong T_2\cong T_3$. However, $C_{3^n}\wr C_3$ is not minimal non-abelian. Hence, $x$ acts on each $T_i$. If $x$ acts non-trivially on $\pcore_p(G)$, then $\pcore_p(G)\langle x\rangle$ is non-abelian and $P=\pcore_p(G)\langle x\rangle$. But then $S$ would be simple. Similarly, if $x$ acts non-trivially on $Q_1:=P\cap T_1$, then $P=Q_1\langle x\rangle$. Write $Q_2:=P\cap T_2=\langle y\rangle$ such that $x^p\in yQ_1$. Then $x$ centralizes $y$. By \cite[Theorem~B]{Gross}, this implies that $x$ induces an inner automorphism on $T_2$. However, $x^p$ induces the inner automorphism by $y$. Hence, $x$ cannot have order greater than $|T_2|_p$. Another contradiction.

\textbf{Case~3:} $S$ is simple.\\
Let $Q:=P\cap S\unlhd P$ be a Sylow $p$-subgroup of $S$. Arguing as in Case~2, we see that $x$ acts non-trivially on $Q$ and therefore $P=Q\langle x\rangle$. First let $Q$ be cyclic. Then $p>2$ and $P$ is metacyclic. Since $\Out(S)$ needs to have an element of order $p$, $S$ must be of Lie type. To obtain \eqref{meta2}, it remains to show that $PS$ is normal in $G$. Assume the contrary. By the structure of $\Out(S)$ (see \cite[Table~5]{Atlas}), $P$ induces a field or graph automorphism of order $p$ on $S$ which acts non-trivially on the subgroup of outer diagonal automorphisms of $S$. In particular, the diagonal automorphism group must have order at least $p+1$, in fact $2p+1\ge 7$ since $p>2$. This excludes all families of simple groups except $S=\PSL^\epsilon(d,q^f)$ where $p\mid f$ and $d\ge 2p+1$. Since $Q$ is cyclic and $f>1$, we have $q\ne p$.
By Fermat's little theorem, $q^{(p-1)f}\equiv q^{2(p-1)f}\equiv 1\pmod{p}$. This contradicts \autoref{pslcyc} (note that $p-1$ is even). Hence, $PS\unlhd G$ and \eqref{meta2} holds.

Let $Q$ be non-cyclic.
Recall that in general $Q$ is homocyclic and $\N_S(Q)$ acts irreducibly on $\Omega(Q)$ (see \cite[Proposition~2.5]{FF}).
This implies that $P$ cannot be metacyclic, as otherwise the fusion in $P$ is controlled by $\N_G(P)$ and $\N_G(Q)=\N_G(P)\C_G(Q)$ acts reducibly on $Q$ according to \autoref{thmfusion}. Hence, let $P\cong\Delta(a,b)$. Then $P'$ is a direct factor of $Q$ and we obtain $Q=\Omega(Q)$. If $Q$ has rank $3$, then $P\cong\Delta(2,1)$. However, by \autoref{thmfusion}, $\N_G(Q)/\C_G(Q)\le\GL(2,p)$ does not act irreducibly on $Q$.
Hence, we may assume that $Q$ has rank $2$. Now $P\cong \Delta(a,1)$ with $a\ge 2$. If $\N_G(P)$ controls the fusion in $P$, then $\N_G(Q)$ would fix $P'$. Hence, we are in Case~\eqref{const2} or \eqref{const3} of \autoref{thmfusion}. Consider $p=2$ first. By Walter's theorem (see \cite[p. 485]{Gorenstein}), $S\cong\PSL(2,q^f)$ with $q^f\equiv\pm3\pmod{8}$. It follows that $f$ is odd and $G/PS\le\Out(S)\le C_{2f}$ by \cite[Table~5]{Atlas}. Here $C_2$ induces a diagonal automorphism and $C_f$ is caused by a field automorphism. So \eqref{sim2} holds. Finally, let $p=3$. Here the claim follows easily from \autoref{simple3}.
\end{proof}

Examples for Theorem~C\eqref{meta2} can be constructed as follows: Let $p>2$ and $a\ge 2$. By Dirichlet's theorem, there exists a prime $q\equiv 1+p^{a-1}\pmod{p^{a+1}}$. Then $q^p\equiv 1+p^a\pmod{p^{a+1}}$ and $S:=\PSL(2,q^p)$ has a cyclic Sylow $p$-subgroup $Q$ of order $p^a$. Let $R\cong C_{p^b}$ and construct $G:=S\rtimes R$ where $R$ acts as the field automorphism $\FF_{q^p}\to\FF_{q^p}$, $\lambda\mapsto\lambda^q$ on $S$. By \cite{Gross}, $R$ acts non-trivially on $Q$ and $P:=Q\rtimes R\cong\Gamma(a,b)$. A different example is $G=\Sz(2^5)\rtimes C_5$ for $p=5$.

\begin{Cor}
Let $G$ be a finite group with a minimal non-abelian Sylow $p$-subgroup and $\pcore_{p'}(G)=1$. Then $G$ has at most one non-abelian composition factor.
\end{Cor}
\begin{proof}
We may assume that $G$ is non-solvable. If $|G|_p=p^3$, then $\F^*(G)$ is quasisimple and $G/\F^*(G)\le\Aut(\F^*(G))\le\Aut(\F^*(G)/\Z(F^*(G))$ is solvable by Schreier's conjecture. Otherwise we have $\F^*(G)=S\times C_{p^b}$ for a simple group $S$ and $b\ge0$ by the proof of Theorem~C. Since $\Aut(C_{p^b})$ is abelian, the claim follows again from Schreier's conjecture.
\end{proof}

\begin{CorD}
The character table of a finite group $G$ determines whether $G$ has minimal non-abelian Sylow $p$-subgroups.
\end{CorD}
\begin{proof}
Let $P$ be a Sylow $p$-subgroup of $G$. We may assume that $\pcore_{p'}(G)=1$.
By \cite[Theorem~B]{NS22}, the character table determines whether $|P:\Z(P)|=p^2$. Suppose that this is the case. By \autoref{lemmna}, it remains to detect whether $|P:\Phi(P)|=p^2$. This is true for $|P|=p^3$, so let $|P|\ge p^4$.
By \autoref{2gen} and \autoref{corpsolv}, we may assume that
$\pcore_p(G)=1$. Now by Theorem~C we expect that $\pcore^{p'}(G)=S\rtimes C_p$ for a simple group $S$ with a cyclic Sylow $p$-subgroup $Q$. As usual, $X(G)$ determines the isomorphism type of $S$. If $Q$ is indeed cyclic, then clearly $P$ is $2$-generated and we are done.
\end{proof}

\section*{Acknowledgment}
We thank Gunter Malle for helpful comments and his careful reading of this paper and Chris Parker for answering some questions on fusion systems. Parts of this work were written while both authors were attending the conference “Counting conjectures and beyond” at the Isaac Newton Institute for Mathematical Sciences in Cambridge. We thank the institute for their financial support via the EPSRC grant no EP/R014604/1.
The first author is supported by the Ministerio de Ciencia e Innovaci\'on (Grant PID2019-103854GB-I00 funded by MCIN/AEI/10.13039/501100011033). The second author is supported by the German Research Foundation (\mbox{SA 2864/1-2} and \mbox{SA 2864/3-1}).

\begin{small}

\end{small}
\end{document}